\documentclass[11pt,reqno, english]{amsart}
\usepackage[a4paper]{geometry}
\usepackage{amsmath,amssymb,amscd,amsthm,amsfonts}
\usepackage{graphicx,subcaption}
\usepackage{hyperref}
\usepackage{dsfont}
\usepackage{xcolor}
\usepackage{color,soul}
\usepackage[utf8]{inputenc}
\usepackage[nobysame, alphabetic]{amsrefs}
\usepackage{tikz}
\usepackage[capitalise]{cleveref}

\newtheorem{theorem}{Theorem}[section]
\newtheorem{corollary}[theorem]{Corollary}
\newtheorem{lemma}[theorem]{Lemma}

\theoremstyle{definition}

\newtheorem{definition}[theorem]{Definition}

\newtheorem{claim}[theorem]{Claim}

\newcommand{\rr}{\mathds{R}}

\newcommand{\ff}{\mathcal{F}}
\DeclareMathOperator{\supp}{supp}

\DeclareMathOperator{\conv}{conv}
\DeclareMathOperator{\aff}{aff}

\title[Variations of Helly's theorem for convex splinters]{Variations of Helly's theorem for convex splinters}

\author[Sober\'on]{Pablo Sober\'on}\address{Baruch College \& The Graduate Center, City University of New York, New York, NY 10010} 
\email{psoberon@gc.cuny.edu}

\author[Ye]{Iris Ye}\address{The University of Chicago Booth School of Business, Chicago, IL 60637} 
\email{iris.ye@chicagobooth.edu}



\thanks{The research of P. Sober\'on was supported by NSF CAREER award no. 2237324 and a PSC-CUNY Trad B award.  This project was conducted as part of the 2024 NYC Discrete Math REU, funded by NSF awards DMS-2051026 and DMS-2349366 and by Jane Street.}

\begin{document}

\begin{abstract}
A convex splinter $K$ is a union of convex sets in $\rr^d$ such that every minimal affine dependent set of $\rr^d$ contained in $K$ is contained in one of the sets.  The study of convex splinters was motivated by the study of flat transversals to convex sets.  We extend several variations of Helly's theorem from convex geometry to convex splinters.  These include fractional and colorful variations of Helly's theorem.  We also extend Tverberg's theorem to convex splinters. 
\end{abstract}

\maketitle

\section{Introduction}

In combinatorial    geometry, the intersection patterns of convex sets is a widely studied subject \cites{Amenta2017, DeLoera2019, Holmsen:2017uf}.  Consider, for example, Helly's theorem, which states that \textit{a finite family of convex sets has a non-empty intersection if and only if every subfamily of at most $d+1$ sets has a non-empty intersection} \cite{Helly:1923wr}.  However, there are other operators similar to the convex hull that also induce interesting combinatorial properties (see, e.g., \cites{Eckhoff:1988eo, Soberon2021, Pirahmad2024}).  In this manuscript, our goal is to extend results from convexity to convex splinters.

\begin{definition}
    A convex splinter of $\rr^d$ is the union of a finite family $K_1,\dots, K_s$ of convex sets of $\rr^d$ so that for any choice $U_1 \subset K_1, \dots, U_s \subset K_s$ of sets of points which are affinely independent, their union $\bigcup_{k=1}^s U_k$ is also affinely independent.
\end{definition}

Convex splinters were introduced by Arocha, Bracho, and Montejano \cites{Arocha2007} as they appear naturally when studying flat transversals to convex sets.  Arocha, Bracho, and Montejano called them \textit{convex partitions}.  We choose a different name to avoid confusion as the term convex partition is commonly used in the context of mass partitions.  The following Helly-type theorem for convex splinters was proved by Arocha et al.

\begin{theorem}[Arocha, Bracho, Montejano 2007 \cite{Arocha2007}]\label{thm:arocha-helly}
    Let \(\ff\) be a finite family of convex splinters in $\rr^d$.  If every \(2d+1\) or fewer sets in \(\ff\) have a non-empty intersection, then the entire family \(\ff\) has a non-empty intersection.
\end{theorem}

The parameter $2d+1$ is optimal.  This is often called the \textit{Helly number} for convex splinters.  Our goal is to extend other results in combinatorial geometry to convex splinters.  Moreover, convex splinters provide a nice geometric setting in which the parameters for the analogues of classic results behave very differently.

We focus on two generalizations of Helly's theorem.  The first is the fractional version by Katchalski and Liu \cite{Katchalski1979} and the second is Lov\'asz's colorful version of Helly's theorem \cite{Barany1982}.  The first studies families of convex sets in which many $(d+1)$-tuples intersect, while the second studies families of convex sets in which certain kinds of $(d+1)$-tuples are required to intersect.  Both are different ways to weaken the condition of Helly's theorem.

There is already a common generalization of the colorful and fractional Helly theorem by B\'ar\'any, Fodor, Montejano, Oliveros, and P\'or \cite{Barany2014}.  Our first result is a version of their result for convex splinters.

\begin{theorem}\label{thm:colorful-fractional-cs}
Let \(d\ge 0\) and let \(\alpha \in (0,1]\). There exists
\(\beta=\beta(\alpha,d)>0\) such that the following holds.  Let \(\ff_1,\ldots,\ff_{d+1}\)
be finite families of convex splinters in \(\rr^d\). Suppose that at least
\(\alpha\prod_{i=1}^{d+1}|\ff_i|\)
of the colorful choices
\( C_i\in\ff_i\), for \(i=1,\ldots,d+1,\)
have non-empty intersection. Then there exists an index \(i\in[d+1]\) and a subfamily
\( \mathcal{G}_i\subseteq\ff_i\) with
\(|\mathcal{G}_i|\ge \beta|\ff_i|\)
such that \( \bigcap_{C\in\mathcal G_i}C\neq\emptyset\).
\end{theorem}

Since convex sets are also convex splinters \cref{thm:colorful-fractional-cs} implies the fractional colorful Helly theorem of B\'ar\'any et al.  Our proof relies on bootstrapping B\'ar\'any et al.'s theorem to cover all convex splinters.

\begin{corollary}[Fractional Helly for convex splinters]\label{coro:frac-splinters}
    Let $d \ge 0, \alpha>0$.  There exists $\beta' = \beta'(\alpha,d)>0$ such that the following holds.  Let $\ff$ be a finite family of convex splinters in $\rr^d$.  Suppose that at least $\alpha \binom{|\ff|}{d+1}$ of the $(d+1)$-tuples of $\ff$ have non-empty intersection.  Then, there exists a subfamily $\mathcal{G} \subset \ff$ with $|\mathcal{G}|\ge \beta'|\ff|$ that has non-empty intersection.
\end{corollary}

The size of the families we need to check, $d+1$, is optimal.  This is known as the fractional Helly number.  Of course, since the Arocha--Bracho--Montejano Helly theorem for convex splinters is optimal, we cannot hope to obtain \(\beta \to 1\) as \(\alpha \to 1\).  

There are now several known Helly-type theorems for which the fractional Helly number is smaller than the corresponding Helly number.  The first example was proven by B\'ar\'any and Matou\v{s}ek for the integer lattice in $\rr^d$ \cite{Barany2003}, where the Helly number is $2^d$ by Doignon's theorem \cites{Doignon1973, Scarf1977, Bell1976/77} but the fractional Helly number is $d+1$.  Similar statements have been proven for \(S\)-Helly numbers \cite{Averkov2012}, discrete Helly for boxes \cites{Edwards2025, Gangopadhyay2025}, and the volumetric Helly theorem \cite{Frankl2025}. Unlike all the examples listed, the proof of \cref{thm:colorful-fractional-cs} does not rely on a modification of the B\'ar\'any--Matou\v{s}ek technique.  We provide another example of this in our discussion of linear partitions in \cref{sec:linear-partitions}.

We also prove a colorful Helly theorem for convex splinters, without the fractional component.  The colorful Helly number we prove is $2d+1$.  This is optimal, as it generalizes \cref{thm:arocha-helly}.

\begin{theorem}[Colorful Helly for convex splinters]\label{thm:colorful-cs}
    Let $\ff_1,\dots,\ff_{2d+1}$ be finite families of convex splinters in $\rr^d$.  If every choice $C_1 \in \ff_1,\dots, C_{2d+1}\in \ff_{2d+1}$ makes an intersecting $(2d+1)$-tuple, then there exists an index $i \in [2d+1]$ such that $\cap \ff_i \neq \emptyset$.
\end{theorem}

An interesting feature of convex splinters is that the intersection of any family of convex splinters in $\rr^d$ is again a convex splinter.  Additionally, the set $\rr^d$ is itself a convex splinter.  Therefore, we can define the convex splinter hull of a set $A \subset \rr^d$ as
\[
\langle A \rangle_{CS} = \cap \{C \subset \rr^d: A \subset C, C \mbox{ is a convex splinter}\}.
\]

Notice that $\langle A \rangle_{CS} \subset \operatorname{conv} A$.  It becomes natural to translate classic results from convex geometry to convex splinters.  We show that one can prove an analogue of Tverberg's theorem \cite{Tverberg:1966tb}.

\begin{theorem}[Tverberg 1966]
    Let $r,d$ be positive integers. Every set of $(r-1)(d+1)+1$ points in $\rr^d$ has at least one partition into $r$ sets whose convex hulls intersect.
\end{theorem}

Tverberg's theorem is a central part of discrete geometry \cite{Barany2018, DeLoera2019}.  The following is a simple version of Tverberg's theorem for convex splinters.

\begin{theorem}\label{thm:tverberg-lp}
    Let $S$ be a set of $(r-1)(2d+1)^2+1$ points in $\rr^d$.  There exists a partition of $S$ into $r$ sets $A_1,\dots, A_r$ such that
    \[
    \bigcap_{j=1}^r \langle A_j\rangle_{CS}\neq \emptyset.
    \]
\end{theorem}

The number of points in \cref{thm:tverberg-lp} is unlikely to be optimal.  However, the dependence on $r$ is linear, which makes it strong enough to prove many of the consequences of Tverberg's theorem.  Since $\langle S \rangle_{CS} \subset \conv (S)$, \cref{thm:tverberg-lp} implies a weak version of Tverberg's theorem.

In \cref{sec:col-frac} and \cref{sec:col-only} we prove \cref{thm:colorful-fractional-cs} and \cref{thm:colorful-cs}, respectively.  We then introduce linear partitions, which allow for simpler proof of the corresponding fractional Helly and colorful Helly analogues in \cref{sec:linear-partitions}.  In \cref{sec:tverberg} we prove \cref{thm:tverberg-lp}.  Finally, we present some applications of our results in \cref{sec:applications}.

\section{Colorful and fractional convex splinters}\label{sec:col-frac}

In this section we prove \cref{thm:colorful-fractional-cs}.  We denote each family $\ff_i$ as a color class.  A set is colorful if it has at most one set from each color.

\begin{proof}[Proof of \cref{thm:colorful-fractional-cs}]
We prove a slightly more general relative statement by induction on the
dimension. Let \(K\) be an \(n\)-dimensional convex set, and let
\( t\ge n+1\).
We claim that for every \(\alpha>0\) there exists
\(\beta=\beta(\alpha,n,t)>0\)
such that the following holds. If
\( \ff_1,\ldots,\ff_t\)
are finite families of convex splinters contained in \(K\), and at least \(\alpha\prod_{i=1}^t|\ff_i|\)
of the colorful \(t\)-tuples have non-empty intersection, then some color class
contains a subfamily of relative size at least \(\beta\) with non-empty total
intersection.

The case \(n=0\) is immediate. In this case \(K\) is a point. A colorful tuple
intersects if and only if all chosen members contain this point. If a positive
fraction of colorful tuples intersects, then at least one color class contains a
positive fraction of members containing the point.

Assume now that \(n\ge 1\), and that the statement is known in dimensions
strictly smaller than \(n\). Let
\(\ff_1,\ldots,\ff_t\)
be finite families of convex splinters contained in \(K\), where \(t\ge n+1\). Let \( N_i=|\ff_i|\).
Suppose that at least \( \alpha\prod_{i=1}^t N_i\)
of the colorful \(t\)-tuples have non-empty intersection.  We split the intersecting colorful tuples into two types.

First suppose that at least \(\frac{\alpha}{2}\prod_{i=1}^t N_i\)
intersecting colorful tuples consist entirely of connected members (i.e., they are convex sets). Let
\( \mathcal C_i\subseteq\ff_i\)
be the subfamily of connected members, and let \( M_i=|\mathcal C_i|\).

We have \( \prod_{i=1}^t M_i\ge \frac{\alpha}{2}\prod_{i=1}^t N_i\).
Since each \(M_i/N_i\le 1\), this implies
\( M_i\ge \frac{\alpha}{2}N_i\) for every \(i\).

The members of the families \(\mathcal C_i\) are ordinary convex sets in the
\(n\)-dimensional affine space \(\operatorname{aff}K\). Moreover, at least an
\(\alpha/2\)-fraction of the colorful choices from
\(\mathcal C_1,\ldots,\mathcal C_t\) have non-empty intersection. By the colorful fractional Helly theorem for
ordinary convex sets, applied with \(t\ge n+1\) colors, there is an index \(i\)
and a subfamily
\( \mathcal G_i\subseteq\mathcal C_i\)
such that
\(|\mathcal G_i|\ge \gamma M_i \) and \( \bigcap_{C\in\mathcal G_i}C\neq\emptyset\),
where \(\gamma=\gamma(\alpha,n,t)>0\). Using the fact that
\( M_i\ge \frac{\alpha}{2}N_i\),
 we obtain
\( |\mathcal G_i|\ge \gamma\frac{\alpha}{2}N_i\).
The desired conclusion holds in this case.

We now assume that at least \( \frac{\alpha}{2}\prod_{i=1}^t N_i\)
intersecting colorful tuples contain a disconnected member. Assign each such
tuple to one disconnected member appearing in it. There are at most \(t\) choices
for the color of the assigned member. For some color \(i_{0}\), there are at least \(
        \frac{\alpha}{2t}\prod_{k=1}^t N_k
\)  tuples
assigned to color \(i_{0}\).

Averaging over the \(N_{i_{0}}\) or fewer members of \(\ff_{i_{0}}\), there exists a
disconnected convex splinter \(C\in\ff_{i_0}\)
such that at least \(\frac{\alpha}{2t}\prod_{k\ne {i_0}} N_k\)
colorful choices from the remaining color classes satisfy
\( \displaystyle C\cap \bigcap_{k\ne {i_0}}C_k\neq\emptyset\), where $C_k \in \ff_k$ for all $k \neq i_0$.

We denote the connected components of \(C\) as
\(C=K_1\cup\cdots\cup K_s\), for \(s\ge 2\).
Let \(n_j=\dim K_j\). Since the components of \(C\) are in general position inside the
\(n\)-dimensional affine space \(\operatorname{aff}K\), we have \( \sum_{j=1}^s(n_j+1)\le n+1\).
Because \(s\ge 2\), this implies \(n_j+1\le n\)
for every \(j\).

Now we restrict all remaining color classes to \(C\):
\( \mathcal{H}_k=\{C\cap D:D\in\ff_k\}\), for \(k\ne i_0\). Previously we showed \( \displaystyle C\cap \bigcap_{k\ne {i_0}}C_k\neq\emptyset\), which is equivalent to \( \displaystyle \bigcap_{k\ne {i_0}}(C \cap C_k)\neq\emptyset\), so among the colorful choices from these \(t-1\) restricted color classes, at least
an
\( \frac{\alpha}{2t}\)
fraction have non-empty intersection.

Every such intersection lies in at least one component \(K_j\). Assign each
intersecting colorful tuple to one component in which it intersects. For some
component \(K_j\), at least an \( \frac{\alpha}{2ts}\)
fraction of all colorful choices from the remaining color classes have non-empty intersection inside \(K_j\). {Since $\sum_{j=1}^s (n_j+1) \le n+1$, and $n_j + 1 \ge 1$, we have $s \le n+1$.  This implies $\frac{\alpha}{2ts} \ge \frac{\alpha}{2t(n+1)}$.}

Restricting further to \(K_j\), we obtain \(t-1\) color classes of convex
splinters contained in \(K_j\). Since \(\dim K_j=n_j\le n-1\)
and \(t-1\ge n\ge n_j+1\),
the induction hypothesis applies in \(K_j\) to these \(t-1\) color classes.
Therefore, for some \(k\ne i_0\), there is a subfamily
\(\mathcal G_k\subseteq\ff_k\)
with \(|\mathcal G_k|\ge \beta'|\ff_k|\)
such that \(\bigcap_{D\in\mathcal G_k}(D\cap C\cap K_j)\neq\emptyset\).
In particular,
\(\bigcap_{D\in\mathcal G_k}D\neq\emptyset\).
This proves the desired conclusion in the second case as well.

Finally, taking \(K=\rr^d\) and \(t=d+1\) proves the
theorem.
\end{proof}

\section{Colorful Helly for convex splinters}\label{sec:col-only}

Before proving \cref{thm:colorful-cs}, we introduce a parameter for convex splinters.

\begin{definition}
    Let \(A =K_1\cup\cdots\cup K_s\) be a convex splinter decomposed into its connected components.  Then we define
    \[
    \rho(A) = \sum_{j=1}^s (2\dim(K_j)+1). 
    \]
\end{definition}

\begin{proof}[Proof of \cref{thm:colorful-cs}]
We prove a slightly more general statement. Let \(A\) be a convex splinter,
and let \(\ff_1,\ldots,\ff_m\) be finite families of convex
splinters contained in \(A\). We claim that if \( m\ge \rho(A)\)
and every colorful choice from \(\ff_1,\ldots,\ff_m\) has
non-empty intersection, then some color class \(\ff_i\) has non-empty
total intersection.

We prove this by induction on \(\rho(A)\).  If \(\rho(A)=1\), then \(A\) is a point and the conclusion follows immediately.  Now suppose that
\( A=K_1\cup\cdots\cup K_s\)
is disconnected.  Let $A=K_1\cup \dots \cup K_s$ with $s \ge 2$, and let $\rho_j = \rho (K_j)$.

Assume, toward a contradiction, that every color class has empty intersection in \(A\). Then, for every \(i\in[m]\) and every component \(K_j\), we also have
\(\bigcap_{C\in\ff_i}(C\cap K_j)=\emptyset\).

Since \(m\ge \rho(A) = \rho_1+ \dots + \rho_s\), we may choose pairwise disjoint sets of color indices
\( I_1,\ldots,I_s\subset[m]\)
such that \(|I_j|=\rho_j\)
for each \(j\).

Fix \(j\). Consider the restricted color classes
\[
        \ff_i|_{K_j}=\{C\cap K_j:C\in\ff_i\},
        \qquad i\in I_j.
\]
These are finite families of convex splinters contained in \(K_j\). Since
\(\rho(K_j)=\rho_j\), the induction hypothesis applied inside \(K_j\) implies
that the colorful Helly conclusion holds for these \(\rho_j\) color classes.
But every color class has empty intersection inside \(K_j\). Therefore, there is a colorful choice with empty intersection.  In other words, we may choose
\( C_i\in\ff_i\), for \(i\in I_j\) such that
\(K_j\cap \bigcap_{i\in I_j}C_i=\emptyset\).

We do this independently for every component \(K_j\). For the remaining colors, choose arbitrary members. The resulting full colorful choice has empty
intersection inside every component \(K_j\), and therefore has empty
intersection in \(A\). This contradicts the hypothesis, which is the contradiction we wanted.

Now consider the case where \(A\) is connected. Then \(A\) is a convex
set of some dimension \(n\), and
\(\rho(A)=2n+1\).  Let \(m\ge 2n+1\), and suppose that every colorful choice from
\(\ff_1,\ldots,\ff_m\) has non-empty intersection.

If every member of every color class is connected, then every member is an
ordinary convex set in the \(n\)-dimensional affine space \(\operatorname{aff}A\).  By the classical
colorful Helly theorem for convex sets, one of these \(m \ge 2n+1 \ge n+1\) color classes has
non-empty intersection.

We may therefore assume that some color class contains a disconnected convex
splinter. Assume without loss of generality that
\( C\in\ff_m\) is disconnected. Write its connected components as
\( C=K_1\cup\cdots\cup K_s\), with \( s\ge 2
\).  Let \( n_j=\dim K_j\).

Since the components of \(C\) are in general position inside the \(n\)-dimensional
affine space \(\operatorname{aff}A\), we have \(\sum_{j=1}^s(n_j+1)\le n+1\).
Therefore
\[
        \rho(C)
        =
        \sum_{j=1}^s(2n_j+1)
        =
        2\sum_{j=1}^s(n_j+1)-s
        \le
        2(n+1)-2
        =
        2n.
\]
Since \(m\ge 2n+1\), we have \( m-1\ge \rho(C)\).

Now restrict the first \(m-1\) color classes to \(C\):
\[
        \mathcal G_i=\{D\cap C:D\in\ff_i\},
        \qquad i=1,\ldots,m-1.
\]
Every colorful choice from \(\mathcal G_1,\ldots,\mathcal G_{m-1}\) has
non-empty intersection.

By the induction hypothesis applied to \(C\), some restricted color class has non-empty intersection, as we wanted to show.

Taking \(A=\mathbb R^d\), for which \(\rho(A)=2d+1\), finishes the proof.

The optimality follows from \cref{thm:arocha-helly}. If a colorful Helly theorem held with fewer than
\(2d+1\) color classes, then by taking all color classes to be equal to some family $\ff$, we  would obtain an ordinary Helly theorem for convex splinters with fewer than \(2d+1\) sets, contrary
to the optimality \cref{thm:arocha-helly}.
\end{proof}

\section{Colorful Helly for linear partitions}\label{sec:linear-partitions}

A family of sets related to convex splinters that is easier to describe is linear partitions, also introduced by Arocha et al. \cites{Arocha2007, Arocha2011}.

\begin{definition}
    A linear partition of $\rr^d$ is the union of a finite family $A_1,\dots, A_s$ of affine spaces of $\rr^d$ so that for any choice $U_1 \subset A_1, \dots, U_s \subset A_s$ of sets of points which are affinely independent, their union $\bigcup_{j=1}^s U_j$ is also affinely independent.
\end{definition}

The Helly number for the family of linear partitions of $\rr^d$ is $\left\lfloor 3(d+1)/2 \right\rfloor$.  This is still greater than $d+1$, but it is smaller than the Helly number for convex splinters.

Some of our results for convex splinters have a much easier proof for linear partitions.  For example, to prove \cref{coro:frac-splinters} for linear partitions it suffices to notice that linear partitions are semialgebraic sets of bounded complexity, and therefore Matou\v{s}ek's fractional Helly for families with bounded dual shatter function \cite{Matousek2004} implies that their fractional Helly number is at most $d+1$.  This gives us a second example of a family with a fractional Helly number smaller than its Helly number in which the proof does not rely on the B\'ar\'any--Matou\v{s}ek hypergraph saturation technique.

Computing the dual shatter dimension for linear partitions is an interesting problem.  Before doing this, consider the following parameter.

\begin{definition}\label{def:pohoata}
    Let $X_d$ be the family of linear partitions in $\rr^d$.  We denote by $\tau'(X_d)$ the largest number $m$ for which there exist $m$ intersecting linear partitions $C_1,\dots, C_m$ such that for each $i \in [m]$ there exists a point $x_i$ in all $C_j$ for $j \neq i$ which additionally satisfies $x_i \not\in C_i$.
\end{definition}

\begin{claim}
    For each dimension $d\ge 0$ we have $\tau'(X_d) \le 2d$.
\end{claim}

For $S \subset \rr^d$, let $\langle S\rangle_{LP}$ be the intersection of all linear partitions that contain $S$.

\begin{proof}
    We first bound the size of a certain configuration of points in $\rr^d$.  We say that $B=\{b_1,\dots,b_{\ell}\}\subset \rr^d$ is a \textit{star} if for each $i \in [\ell]$ we have $b_i \not\in \langle B\setminus \{b_i\}\rangle_{LP}$.  Equivalently, there exists a linear partition $D_i$ that contains $B\setminus\{b_i\}$ but not $b_i$.  Let us show by induction on $d$ that the maximum cardinality of a star in $\rr^d$ is at most $2d+1$.  For $d=0$ this is clear.

    Since $D_m=\langle B\setminus \{b_m\}\rangle_{LP}$ is neither empty nor equal to $\rr^d$, it is a union of affine spaces $A_1,\dots, A_s$, each of dimension at most $d-1$.  If $t_j = \dim A_j$ for all $j$, we have $\sum_j (t_j + 1) \le d+1$.

    A key property about stars is that the intersection of a star with an affine space $A$ is either empty or a star in \(A\).

    Now consider the sets $A_1 \cap B, \dots, A_s \cap B$.  By induction, we have $|A_j \cap B| \le 2t_j + 1$.  If $s = 1$, then there are at most $2t_1+1 \le 2(d-1)+1 = 2d-1$ points of $B$ in $D_m$ and one point outside of it, which gives us a bound of $2d$ points for $|B|$.

    Now assume $d \ge 2$.  Since there is at most one point missing from $D_m$, we have
    \[
    |B|-1 \le \sum_{j=1}^s (2t_j+1) = \left(2\sum_{j=1}^s (t_j+1)\right) - s \le 2d+2-s \le 2d.
    \]
    The bound above implies $|B| \le 2d+1$.

    Now we prove the bound on $\tau'(X_d)$.  Let $x_1,\dots, x_m$ and $C_1,\dots,C_m$ be a family of points and linear partitions that realize $\tau'(X_d)$, and let $p$ be a point in the intersection of $C_1,\dots,C_m$.  Let $B = \{x_1,\dots,x_m\}$.

    The linear partition $C_m$ is neither empty nor $\rr^d$, so it is a union of affine spaces $A_1,\dots, A_s$, each of dimension at most $d-1$.

    Assume without loss of generality that $p \in A_s$.  Let $t_j = \dim A_j$ for $j \in [s]$.

    The set $A_s \cap B$ together with the family $\{C_i \cap A_s: x_i \in A_s\}$ is a configuration of points and linear partitions in dimension $t_s$ as in  \cref{def:pohoata}.  Therefore, by induction we have $|B \cap A_s| \le 2t_s$.  For $j \neq s$, the configuration $B \cap A_j$ is a star, which implies $|B \cap A_j| \le 2t_j + 1$.  The set $B \setminus C_m$ has at most one element.

    If $s = 1$, then $|B| -1 \le 2t_1 \le 2(d-1) \le 2d-1 $.  This implies $|B| \le 2d$.

    If $s \ge 2$, then
    \begin{align*}
    |B|-1 \le \left(\sum_{j=1}^{s-1} (2t_j + 1)\right) + (2t_s) = \left(\sum_{j=1}^{s-1} (2t_j + 2)\right)-(s-1) + 2(t_s+1)-2 \\
    \left(\sum_{j=1}^s 2(t_j+1)\right)-(s-1)-2\le 2(d+1)-s-1 \le 2d-1.
    \end{align*}
    Therefore, $|B| \le 2d$, as we wanted to show.
    
\end{proof}

With this result, we can prove \cref{thm:colorful-cs} for linear partitions relatively quickly.  The parameter $\tau'$ was introduced by Pohoata, Yang, and Zhang to bound the colorful Helly number for families of sets.

\begin{proof}[Proof of \cref{thm:colorful-cs} for linear partitions]
    By \cite{Pohoata2025}*{Thm. 1.3}, the number of color classes needed for \cref{thm:colorful-cs} is bounded above by \(\tau'(X_d)+1
    \), which is at most $2d+1$.
\end{proof}

The dual VC dimension of a family of sets is an important parameter.  Let us define it properly.  Let $X$ be a family of subsets of some set $Y$.  Given $m$ sets $C_1,\dots, C_m$ in $X$, two elements $x,y \in Y$ are going to be equivalent if they belong to exactly the same sets among the $C_i$.  Formally,
\[
\{i \in [m]: x \in C_i\}=\{i \in [m]: y \in C_i\}.
\]

The dual VC dimension of $X$, denoted $VC^*(X)$, is the largest value of $m$ for which we can find sets $C_1,\dots, C_m \in X$ that generate the full $2^m$ possible equivalence classes.  In other words, for every subset $A \subset [m]$ there exists a point $y \in Y$ that is exactly in all $C_i$ for $i \in A$ and in none of the $C_i$ for $i \not\in A$.

\begin{theorem}
    Let $d$ be a positive integer and let $X_d$ be the family of linear partitions in \(\rr^d\).  Then,
    \[
    d+1 \le VC^*(X_d) \le 2d.
    \]
\end{theorem}

\begin{proof}
 The upper bound is immediate, $VC^*(X_d) \le \tau'(X_d) \le 2d$.

 For the lower bound, we give a simple construction.  Let $\Delta^d$ be a $d$-dimensional simplex in $\rr^d$, with facets $F_1,\dots,F_{d+1}$, and let $p$ be the barycenter of $\Delta^d$.  Consider the linear partitions $C_i = \operatorname{aff}(F_i) \cup \{p\}$ for $i \in [d+1]$.  The barycenter of every non-empty face of $\Delta^d$ is in the intersection of a different subset of the linear partitions $C_i$.  If we pick any point that is not contained in any $C_i$ we can also realize the empty intersection.
\end{proof}

We conjecture that $VC^*(X_d) = d+1$.

\section{Tverberg-type results}\label{sec:tverberg}

To prove \cref{thm:tverberg-lp}, we first need to prove the convex splinters version of a couple of other results in combinatorial geometry.  The first is a version of Rado's centerpoint theorem \cite{Rado1946}.

\begin{lemma}\label{lem:centerpoint-lp}
Let $S$ be a finite set of points in $\rr^d$.  There exists a point $p$ such that every subset $S' \subset S$ with $|S'|>\left(1-\frac{1}{2d+1}\right)|S|$ satisfies $p \in \langle S'\rangle_{CS}$.    
\end{lemma}

\begin{proof}
    The family 
    
    \[\ff= \left\{\langle S' \rangle_{CS}: S'\subset S, |S'|>\left(1-\frac{1}{2d+1}\right)|S| \right\}\]
    satisfies the intersection property of \cref{thm:arocha-helly}.  Take arbitrary \(C_i=\langle S_i\rangle_{CS}\in\mathcal F\), \(i=1,\ldots,2d+1\). Then $|S\setminus S_i|<\frac{|S|}{2d+1}$, so $\left|\bigcup_{i=1}^{2d+1}(S\setminus S_i)\right|<|S|$. This means that there exists some $p \in S \setminus \bigcup_{i=1}^{2d+1} (S \setminus S_i) = \bigcap_{i=1}^{2d+1} S_i$, so $p \in \bigcap_{i=1}^{2d+1}\langle S_i\rangle_{CS} =\bigcap_{i=1}^{2d+1}C_i.$   This implies that $\ff$ is intersecting.
\end{proof}

The second is a version of Carath\'eodory's theorem for convex splinters.

\begin{lemma}\label{lem:caratheodory-lp}
    If $p$ is a point in $\rr^d$ and $S \subset \rr^d$ is a finite set such that $p \in \langle S \rangle_{CS}$, there exists a set $S' \subset S$ with $|S'| \le 2d+1$ such that $p \in \langle S' \rangle_{CS}$.
\end{lemma}

\begin{proof}
Given a set $S$ of points in $\rr^d$, let us first describe the set $\langle S \rangle_{CS}$.  We are going to make a hypergraph $\mathcal{H}(S)$ with vertex sets $S$.  A subset $A \subset S$ is going to be an edge of the hypergraph if $A$ is affine dependent, but every proper subset of $A$ is affine independent.

Notice that if $C$ is a convex splinter such that $S \subset C$, then $A$ must be in a single connected component of $C$.  Therefore, if $E_1,\dots,E_s$ are the connected components of $\mathcal{H}(S)$, we have $\conv E_i \subset C$. For each \(i\), choose an affine basis $B_i\subset E_i$ of
$\aff(E_i)$. We claim that $B_1\cup\cdots\cup B_s$ is affinely independent. Otherwise, there exists a minimal affinely dependent
subset $A \in B_1\cup\cdots\cup B_s$. Since $A\subset S$, $A$ is an edge of $\mathcal H(S)$. Moreover, $A$ must meet at least two distinct $B_i$'s, because each $B_i$ is affinely independent, contradicting the fact that $E_i$'s are connected components.

This shows the affine spaces $\operatorname{aff}(E_1), \dots, \operatorname{aff}(E_s)$ form a linear partition.  Therefore
\[
\langle S \rangle_{CS} = \conv (E_1) \cup \dots \cup \conv (E_s).
\]

Now suppose $p \in \langle S \rangle_{CS}$.  By the description above, there is a connected component $E$ of $\mathcal{H}(S)$ such that $p \in \conv E$.  By Carath\'eodory's theorem we have a set $I \subset E$ such that $|I| \le d+1$ and $p \in \conv I$.  We may assume that $I$ is affine independent, after removing points if necessary.

Let $r = \dim \operatorname{aff}(E)$.  Since $I \subset E$ and it is affine independent, there must exist an affine independent set $B$ such that $I \subset B \subset E$ such that $|B| = r+1$. 

For each point $x \in E \setminus B$, the set $B \cup \{x\}$ is affine dependent, and there is a unique  affine dependent subset of $B \cup \{x\}$.  We denote it by
\[
C_x = \{x\} \cup \supp_B(x).
\]
In other words, $\supp_B(x)$ is the set of vertices of $B$ that have non-zero coefficients in an affine combination of $B$ that gives $x$.

Now we define a graph $G$ on the vertex set $B$.  Two points $u,v \in B$ will form an edge if there exists some $x_{uv} \in E \setminus B$ such that
\[
\{u,v\}\subset \supp_B(x_{uv}).
\]

Let us show that $B$ is connected.  If $G$ was disconnected, then there is a nontrivial partition $B = B_1 \cup B_2$ without edges between the two parts.  Every point $x \in E \setminus B$ would have $\supp_B(x) \subset B_1$ or $\supp_B(x) \subset B_2$.  This induces a partition of $E$ into two sets $E_1$ and $E_2$.  In particular, $E_1 \subset \aff (B_1)$ and $E_2 \subset \aff (B_2)$.

This means that no edge of $\mathcal{H}(S)$ meets both $E_1$ and $E_2$, contradicting the fact that $E$ was a connected component of $\mathcal{H}(S)$.  Therefore, $G$ is connected.

Let $T$ be a spanning tree of $G$, and let $X = \{x_{uv}: uv \mbox{ is an edge of }T\}$.  Since $T$ has $|B|-1=r$ edges, $|X| \le r$.  Now let $S' = B \cup X$.  We have $|S'|\le 2r+1 \le 2d+1$.

If we look at $\langle S' \rangle_{CS}$, for every two points $u,v \in B$, we have that $\{x_{uv}\}\cup \supp_B(x_{uv})$ is a minimal affine dependence that contains both $u$ and $v$.  Therefore, $T \subset \mathcal{H}(S')$, and in particular $B$ must be in a single connected component of $\mathcal{H}(S')$.  This finishes the proof as
\[
p \in \conv (I) \subset \conv (B) \subset \langle S' \rangle_{CS}.
\]
\end{proof}

The result above is optimal, since we can take $S = \{0,e_1,\dots,e_d,2e_1,\dots,2e_d\}$, where $e_1,\dots,e_d$ are the canonical basis of $\rr^d$.  We can also take $p = (e_1+\dots+e_d)/2d$.  For this example, we have $|S|=2d+1$ and $p \in \langle S \rangle_{CS} = \conv(S)$.  However, for every $s \in S$ we can verify that $p \not\in \langle S \setminus \{s\}\rangle_{CS}$.

Now we are ready to prove \cref{thm:tverberg-lp}.

\begin{proof}
We construct $A_1,\dots,A_r$ inductively.  We claim we can find a point $p$ and $A_1 \subset S$ such that $p \in \langle A_1 \rangle_{CS}$ and $|A_1| \le 2d+1$. By \cref{lem:centerpoint-lp}, we can choose $p$ with the stated centerpoint property. Since $S$ itself satisfies $|S|>\left(1-\frac1{2d+1}\right)|S|$, we get $p\in \langle S\rangle_{CS}$. By \cref{lem:caratheodory-lp}, there exists $A_1\subset S$ such that $p\in \langle A_1\rangle_{CS}$ and $|A_1|\le 2d+1$.

For the inductive step, suppose we have constructed $j$ subsets $A_1,\dots, A_j \subset S$ that are pairwise disjoint such that $p \in \langle A_i \rangle_{CS}$ and such that $|A_i| \le 2d+1$ for all $i$.  Let $S' = S \setminus (A_1 \cup \dots \cup A_j)$.  Then, as long as $j< r$, we have
\begin{align*}
    |S'|=|S \setminus (A_1 \cup \dots \cup A_j)| \ge |S| - (r-1)(2d+1) > \left(1-\frac{1}{2d+1}\right)|S|.
\end{align*}
Therefore, $p \in \langle S'\rangle_{CS}$.  By \cref{lem:caratheodory-lp}, there exists a set $A_{j+1} \subset S'$ such that $p \in \langle A_{j+1}\rangle_{CS}$ and $|A_{j+1}|\le 2d+1$.  We can continue this process until we have $r$ sets.

Finally, distribute the remaining points of \(S\) arbitrarily among the sets
\(A_1,\ldots,A_r\). Since the convex splinter hull is monotone, the common
point \(p\) remains in all \(r\) hulls.
\end{proof}

\section{Applications}\label{sec:applications}

\subsection{Flat transversals}

We briefly explain how Theorem~\ref{thm:colorful-fractional-cs} can be interpreted as a fractional Hadwiger-type theorem for flat transversals to convex sets. This is the geometric motivation that led Arocha, Bracho, and Montejano to work with convex splinters \cite{Arocha2007}.  We refer the reader to the paper by Arocha et al. for the precise details of their reduction.

The key observation is that, even though there is no Helly-type theorem for $m$-flats transversals to convex sets, if we fix some anchor sets $K_0,\dots, K_{m+1}$ that the transversal must meet, then such theorems exist.

Let \(m,n\ge 0\). Suppose that
\[
K_0,K_1,\ldots,K_{m+1}
\]
are \(n\)-dimensional convex sets whose affine hulls are in sufficiently general position (i.e., their union is a linear partition). The \(m\)-flats meeting all of these \(m+2\) sets form a ruling. After choosing one of the admissible order-type regions of this ruling, the relevant transversal \(m\)-flats can be parametrized by an \(n\)-dimensional affine space, say \(X\).

Now let \(K\) be another \(n\)-dimensional convex set such that the union of $K, K_0,\dots,K_{m+1}$ is a convex splinter. Consider the set
\[
C(K)\subset X
\]
of parameters \(x\in X\) such that the corresponding \(m\)-flat \(Y_x\) meets \(K\) with the prescribed order-type compatibility. Arocha, Bracho, and Montejano showed that \(C(K)\) is a convex splinter in \(X\). Thus a collection of convex sets
\[
K^{(1)},\ldots,K^{(r)}
\]
has a common compatible transversal \(m\)-flat, together with the fixed base sets \(K_0,\ldots,K_{m+1}\), exactly when
\[
C(K^{(1)})\cap \cdots \cap C(K^{(r)})\neq \emptyset.
\]
In other words, the transversal problem becomes an intersection problem for convex splinters in the \(n\)-dimensional parameter space \(X\).

Applying Theorem~\ref{thm:colorful-fractional-cs} gives the following fractional colorful transversal theorem.

\begin{corollary}[Colorful fractional theorem for compatible flat transversals]
Let \(m,n\ge 0\), and fix \(m+2\) convex sets \(K_0,K_1,\ldots,K_{m+1}\) of dimension \(n\) whose union is a convex splinter, together with an admissible order type for the corresponding transversal \(m\)-flats.

Let \(\ff_1,\ldots,\ff_{n+1}\) be finite families of additional \(n\)-dimensional convex sets such that the union of $K_0,\dots,K_{m+1}, F_i$ is a convex splinter for all $F_i \in \ff_i$. Suppose that at least
\(\alpha\prod_{i=1}^{n+1}|\ff_i|\)
colorful choices
\(
L_i\in \ff_i\), for \(i=1,\ldots,n+1,\)
have a compatible transversal \(m\)-flat meeting \(K_0,K_1,\ldots,K_{m+1},L_1,\ldots,L_{n+1}\).
Then, there exist an index \(i\in[n+1]\), a subfamily
\(\mathcal{G}_i\subset \ff_i\)
with \(|\mathcal G_i|\ge \beta|\ff_i|\),
where \(\beta=\beta(\alpha,n)>0\), and a single \(m\)-flat \(Y\) such that \(Y\) meets every set in
\(\{K_0,K_1,\ldots,K_{m+1}\}\cup \mathcal{G}_i\)
with the prescribed order-type compatibility.
\end{corollary}

For the state of the art regarding $m$-flat transversals to convex sets, see \cite{McGinnis2026} and the references therein.

\subsection{Eigenvectors for linear maps}

A natural way to find linear partitions is using eigenspaces of linear functions in $\rr^d$.

Let $L : \rr^d \to \rr^d$.  Let $E$ be the set of eigenvectors of $L$.  For a random hyperplane $A$ that does not contain the origin, $A \cap E$ will be a linear partition in $A$.  This was observed by Polyanskii \cite{polyanskii2016helly}.

A direct application of \cref{coro:frac-splinters} gives the following.

\begin{corollary}
    Let $d$ be a positive integer.  For every $\alpha>0$ there exists $\beta=\beta(\alpha,d)>0$ such that the following holds.  Let $L_1,\dots, L_n$ be linear maps from $\rr^d$ to $\rr^d$.  Suppose that at least $\alpha \binom {n}{d}$ of the $d$-tuples of linear maps share an eigenvector.  Then, there is a subfamily of at least $\beta n$ linear maps that share an eigenvector.
\end{corollary}


\end{document}